\documentclass[10pt]{amsart}
\pdfoutput=1
\usepackage{amsmath}
\usepackage{amsfonts}
\usepackage{amssymb}
\usepackage{listings}
\usepackage{amsthm}
\usepackage{hyperref}
\usepackage{amsmath,bbding}
\usepackage{enumerate}
\usepackage{comment}
\usepackage{float}
\usepackage{youngtab}
\usepackage{ytableau}
\usepackage{rotating}
\usepackage{xcolor}
\usepackage{bbding}
\usepackage{tikz}
\usepackage{float}
\usepackage{tikz-cd}
\usepackage[paper=a4paper, margin=3.5cm]{geometry}

\def\ZZ{{\NZQ Z}}

\def\frk{\frak}               

\def\Phi{{\frk n}}
\def\Phi{{\frk N}}
\def\opn#1#2{\def#1{\operatorname{#2}}} 
\opn\chara{char} \opn\length{\ell} \opn\pd{pd} \opn\rk{rk}
\opn\projdim{proj\,dim} \opn\injdim{inj\,dim} \opn\rank{rank}
\opn\spn{span}\opn\Seg{Seg}
\opn\depth{depth} \opn\grade{grade} \opn\height{height}
\opn\embdim{emb\,dim} \opn\codim{codim}

\opn\Tr{Tr} \opn\bigrank{big\,rank}
\opn\superheight{superheight}\opn\lcm{lcm}
\opn\trdeg{tr\,deg}
\opn\reg{reg} \opn\lreg{lreg} \opn\ini{in} \opn\lpd{lpd}
\opn\size{size}\opn\bigsize{bigsize}
\opn\cosize{cosize}\opn\bigcosize{bigcosize}
\opn\sdepth{sdepth}\opn\sreg{sreg}
\opn\link{link}\opn\fdepth{fdepth} \opn\trdeg{trdeg} \opn\mod{mod}
\opn\spann{span}
\opn\div{div} \opn\Div{Div} \opn\cl{cl} \opn\Cl{Cl}
\opn\Spec{Spec} \opn\Supp{Supp} \opn\supp{supp} \opn\Sing{Sing}
\opn\Ass{Ass} \opn\Min{Min}\opn\Mon{Mon} \opn\dstab{dstab} \opn\astab{astab}
\opn\Syz{Syz}
\opn\Ann{Ann} \opn\Rad{Rad} \opn\Soc{Soc} \opn\Aut{Aut}
\opn\Im{Im} \opn\Ker{Ker} \opn\Coker{Coker} \opn\Am{Am}
\opn\Hom{Hom} \opn\Tor{Tor} \opn\Ext{Ext} \opn\End{End}
\opn\Aut{Aut} \opn\id{id}

\opn\nat{nat}
\opn\pff{pf}
\opn\Pf{Pf} \opn\GL{GL} \opn\SL{SL} \opn\mod{mod} \opn\ord{ord}
\opn\Gin{Gin} \opn\Hilb{Hilb}\opn\sort{sort}
\opn\S{S} \opn\dim{dim} \opn\supp{supp}\opn\trdeg{trdeg}\opn\sort{sort}
\opn\aff{aff} \opn\con{conv} \opn\relint{relint} \opn\st{st}
\opn\lk{lk} \opn\cn{cn} \opn\core{core} \opn\vol{vol}
\opn\link{link} \opn\star{star}\opn\lex{lex}
\opn\conv{conv} \opn\Ehr{Ehr}\opn\Pic{Pic}
\opn\Conv{Conv}
\opn\gr{gr}

\def\pot#1#2{#1[\kern-0.28ex[#2]\kern-0.28ex]}

\opn\dirlim{\underrightarrow{\lim}}
\opn\inivlim{\underleftarrow{\lim}}
\let\tensor=\otimes

\def\Implies{\ifmmode\Longrightarrow \else
        \unskip${}\Longrightarrow{}$\ignorespaces\fi}
\def\implies{\ifmmode\Rightarrow \else
        \unskip${}\Rightarrow{}$\ignorespaces\fi}
\def\iff{\ifmmode\Longleftrightarrow \else
        \unskip${}\Longleftrightarrow{}$\ignorespaces\fi}

\let\:=\colon

\newtheorem{Theorem}{Theorem}[section]
 \newtheorem{Lemma}[Theorem]{Lemma}
 \newtheorem{Corollary}[Theorem]{Corollary}
 
 \newtheorem{Remark}[Theorem]{Remark}

 \newtheorem{Conjecture}[Theorem]{Conjecture}
 
 \newtheorem{Computation}[Theorem]{Computation}

\def\ZZ{\mathbb{Z}}

\begin{document}
 \title {Classification of normal phylogenetic varieties for tripods}
\keywords {group-based model, projective variety, polytope, normal}
 
 \author{Rodica Andreea Dinu}
\address{%
	University of Konstanz, Fachbereich Mathematik und Statistik, Fach D 197 D-78457 Konstanz, Germany, and Simion Stoilow Institute of Mathematics of the Romanian Academy, Calea Grivitei 21, 010702, Bucharest, Romania}
	\email{rodica.dinu@uni-konstanz.de}

\author{Martin Vodi\v{c}ka}
\address{Šafárik University, Faculty of Science, Jesenná 5, 04154 Košice, Slovakia}
	\email{martin.vodicka@upjs.sk}

\maketitle

 \begin{abstract} 
 We provide a complete classification of normal phylogenetic varieties coming from tripods, and more generally, from trivalent trees. Let $G$ be an abelian group. We prove that the group-based phylogenetic variety $X_{G,\mathcal{T}}$, for any trivalent tree $\mathcal{T}$, is projectively normal if and only if $G\in \{\ZZ_2, \ZZ_3, \ZZ_2\times\ZZ_2, \ZZ_4, \ZZ_5, \ZZ_7\}$.
 \end{abstract}

 \maketitle

\section{Introduction}
Phylogenetics aims to investigate the evolution of species over time and to determine the genetic relationships between species based on their DNA sequences, \cite{fels}. A correspondence is established to highlight the differences between the DNA sequences and this is useful to reconstruct the relationships between species, \cite{kimura}.
The ancestral relationships can be encoded by the structure of a tree, which is called a \textit{phylogenetic tree}. Phylogenetics reveals connections with several parts of mathematics such as algebraic geometry \cite{bw, erss}, combinatorics \cite{matJCTA, RM2} and representation theory \cite{manon}. We consider the algebraic variety associated to a phylogenetic tree, called a \textit{phylogenetic variety}. The construction of this variety will be presented in this article and it can be consulted also in \cite{erss} for more details. For group-based models, this variety is toric \cite{HP,evans, sz}. Algebraic and geometric properties for these varieties are presented in \cite{RM, RM2, h-rep, mateusz, mateuszKimura, ss, martinko}. 

Normality is a very important property, as most of the results in toric geometry work only for normal varieties. The reader may consult \cite{brunsg, cox, fulton, binomialideals, sturmfels} and the references therein. A polytope $P$ whose vertices generate lattice $L$ is {\it normal} if every point in $kP\cap L$ can be expressed as a sum of $k$ points from $P\cap L$. The normality of polytope implies the projective normality of the associated projective toric variety. We are interested in understanding the normality property of group-based phylogenetic varieties; hence, the normality of their associated lattice polytopes.

A \textit{tripod} is a tree that has exactly one inner node and 3 leaves, and a \textit{trivalent tree} is a tree for which each node has vertex degree $\leq 3$.
By Theorem~\ref{key}  and \cite[Lemma 5.1]{mateusz}, it follows that, for a given group, in order to check the normality for any trivalent tree, it is enough to check the normality for the chosen group and the tripod. More generally, if one wants to check the normality for the algebraic variety for this group and any tree, it is enough to verify the normality for claw trees.

In addition, by \cite[Remark 2.2]{RM}, non-normality for tripods gives non-normality for any non-trivial tree (i.e. a tree that is not a path). 

Hence, it is important to understand the normality when the phylogenetic trees are tripods. Actually, the polytope $P_{G,3}$ associated to the tripod, encodes the group multiplication, which makes it an interesting object to study even without a connection to phylogenetics.

Buczy\'nska and Wi\'sniewski \cite{bw} proved that the toric variety associated to any trivalent tree and the group $\mathbb{Z}_2$ is projectively normal. The same result holds actually for any tree by using \cite{martinko}, where the normality was proved for the $3$-Kimura model (i.e. when $G=\ZZ_2\times \ZZ_2$). Also, computations from \cite{mateusz} show that $X_{\ZZ_4,\mathcal{T}}$ is normal for any trivalent tree $\mathcal{T}$.

In \cite[Proposition 2.1]{RM}, it was shown that, if $G$ is an abelian group of even cardinality greater or equal to 6, then the polytope $P_{G, 3}$ is not normal. As a consequence, in this case, the polytope $P_{G,n}$ for any $n$-claw tree is not normal, and hence, the algebraic variety $X_{G,n}$ representing this model is not projectively normal.

Hence, not every group-based phylogenetic variety is normal. Thus, there is a need for a classification.

We present now the structure of the article. In Section~\ref{prel}, we introduce the algebraic variety associated to a group-based model and its associated polytope. We provide also some results, known in the literature, such as the vertex description for $P_{G,3}$ and how a trivalent tree can be obtained by gluing tripod trees, and more generally, how an arbitrary tree can be obtained by gluing $n$-claw trees (i.e. trees that have one inner node, and $n$ leaves). Section~\ref{non-normal} is devoted to proving that the algebraic variety $X_{G,3}$ is not projectively normal for any abelian group having its order an odd number greater or equal to 11. In order to show the non-normality of a polytope, it is enough to find a lattice point that belongs to the $(mk)$-th dilation of the polytope in the corresponding lattice, but not in the $k$-th dilation of the polytope, again, in the corresponding lattice, for some integers $k,m$. Our strategy is to provide such examples for all abelian groups having the order an odd number $\geq 11$. For this, we consider a cubic graph $\Gamma=(E,V)$ whose edges are colored blue, yellow, and red such that no two adjacent edges have the same color. We call a function $f\:\; E \rightarrow \Gamma$ \textit{good} if, to any vertex, when taking the sum of the adjacent blue, yellow, and red vertices through function $f$ we obtain 0. In Lemma~\ref{condition}, and later, in Lemma~\ref{condition2} which is the key of our proofs from this section, we give the sufficient conditions that should be satisfied by a 3-colorable graph. that it is not bipartite, and a good function in order to obtain a lattice point that has required the property, which destroys the normality of the polytope. In Theorem~\ref{43}, we provide a suitable graph and we show the existence of a good function, showing that if $G$ is an abelian group of odd order which is greater than 43, then $P_{G,3}$ is not normal. If the order of $G$ is an odd number between 12 and 43, we prove in Theorem~\ref{odd11}, that again, the polytope $P_{G,3}$ is not normal, this time by proving concrete examples of good functions. In Theorem~\ref{11}, we show that $P_{\ZZ_{11}, 3}$ is not normal, but we find a larger graph with a good function that satisfies the properties from Lemma~\ref{condition2}.
In Section~\ref{comp}, we present some computational results and the code we used to obtain those results. Computation~\ref{5and7} shows that $P_{\ZZ_5, 3}$ and $P_{\ZZ_7, 3}$ are normal, while Computation~\ref{9and3x3} shows that $P_{\ZZ_9, 3}$ and $P_{\ZZ_3\times \ZZ_3, 3}$ are not-normal. 
Section~\ref{main} is devoted to the main result of this article, Theorem~\ref{class}, which unifies all the results obtained apriori and provides a complete classification of the group-based phylogenetic varieties for tripods:\\

\textbf{Theorem.}
Let $G$ be an abelian group. Then the polytope $P_{G,3}$ associated to a tripod and the group $G$ is normal if and only if $G\in\{\ZZ_2, \ZZ_3, \ZZ_2\times \ZZ_2, \ZZ_4, \ZZ_5, \ZZ_7\}$. Moreover, this result holds for the polytope $P_{G,\mathcal{T}}$ associated to any trivalent tree.\\

As a consequence, we obtain that:\\

\textbf{Corollary.}
Let $G$ be an abelian group. Then the group-based phylogenetic variety $X_{G,3}$ associated to a tripod and the group $G$ is projectively normal if and only if $G\in\{\ZZ_2, \ZZ_3, \ZZ_2\times \ZZ_2, \ZZ_4, \ZZ_5, \ZZ_7\}$.
Moreover, this result holds for the group-based phylogenetic variety $X_{G,\mathcal{T}}$ associated to any trivalent tree.


\section{Preliminaries}\label{prel}

In this section, we introduce the notation and preliminaries that will be used in the article. The reader may consult also \cite{ss, h-rep, RM}.

\subsection{The algebraic variety associated to a group-based model} We present here the construction of the algebraic variety associated to a model. 
A \textit{phylogenetic tree} is a simple, connected, acyclic graph that will come together with some statistical information. We will denote its vertices by $V$ and its edges by $E$. A vertex $v$ is called a \textit{leaf} if it has valency 1, and all the vertices that are not leaves will be called \textit{nodes}. The set of all leaves will be denoted by $\mathcal{L}$. The edges of $T$ are labeled by transition probability matrices $\mathcal{M}$ which show the probabilities of changes of the states from one node to another. A \textit{representation} of a model on a phylogenetic tree $T$ is an association $E \rightarrow \mathcal{M}$. We denote the set of all representations by $\mathcal{R}(T)$. Each node of $T$ is a random variable with $k$ possible states chosen from the state space $S$. To each vertex $v$ of $T$ we associate an $|S|$-dimensional vector space $V_v$ with basis $(v_s)_{s\in S}$. An element of $\mathcal{M}$ associated to $e:=(v_1, v_2) \in E$ may be viewed as an element of the tensor product $V_{v_1}\tensor V_{v_2}$. We fix a representation $M\in \mathcal{R}(T)$ and an association $a\:\; \mathcal{L} \rightarrow S$. Then the probability of $a$ may be computed as follows:
\[
P(M, a)= \sum_{\sigma} \prod_{(v_1, v_2)\in E} (M(v_1, v_2))_{(\sigma(v_1), \sigma(v_2))},
\]
where the sum is taken over all associations $\sigma \:\; V \rightarrow S$ such that their restrictions to $\mathcal{L}$ coincide to $a$. By identifying the association $a$ with a basis element $\tensor_{l\in \mathcal{L}} l_{a(l)} \in \bigotimes_{l\in \mathcal{L}}V_l$, we get the map:
\[
\Theta\:\; M \rightarrow \sum_{a} (P(M\tensor a)\tensor_{l\in \mathcal{L}} l_{a(l)})\in \bigotimes_{l\in \mathcal{L}} V_l.
\]
The Zariski closure of this map is an algebraic variety that represents the model and we call it a \textit{phylogenetic variety}. For group-based models, we denote this variety by $X_{G,T}$, where $G$ is the group representing the model and $T$ is the tree as above. We call it a \textit{group-based phylogenetic variety}.

\subsection{The polytope associated to a group-based phylogenetic variety}

For special classes of phylogenetic varieties, Hendy and Penny \cite{HP} and, later, Erdős, Steel, and Székely \cite{sz}, used the Discrete Fourier Transform in order to turn the map $\Theta$ into a monomial map. 
In particular, it is known that the group-based phylogenetic variety $X_{G,T}$ is a \textit{toric variety}. Hence, one can use toric methods when working with it, because the geometry of a toric variety is completely determined by the combinatorics of its associated lattice polytope. We denote the polytope associated to the projective toric variety $X_{G,T}$ by $P_{G,T}$. When $T$ is the $n$-claw tree. we denote the corresponding polytope by $P_{G,n}$. In this paper, we will always work with 3-claw trees, which are also called \textit{tripods}, and we will denote the corresponding polytopes by $P_{G,3}$.

\subsection{Vertex description of $P_{G,3}$}

We introduce some notation that will be used when working with polytopes $P_{G,3}\subset\mathbb R^{3|G|}$. We label the coordinates of a point $x\in\mathbb R^{3|G|}$ by $x_g^j$, where $1\le j\le 3$ corresponds to an edge of the tree, and $g\in G$ corresponds to a group element. For any point $x\in \mathbb Z_{\ge 0}^{3|G|}$, we define its $G$-presentation as an $3$-tuple $(G_1, G_2,G_3)$ of multisets of elements of $G$. Every element $g\in G$ appears exactly $x_g^j$ times in the multiset $G_j$. We denote by $x(G_1,G_2,G_3)$ the point with the corresponding $G$-presentation. 

The vertex description of the polytope $P_{G,3}$ (and, more generally, for $P_{G,n}$) is known and may be consulted in \cite{bw, mateusz, ss}. We recall this description in terms of $G$-presentations.

\begin{Theorem} 
The vertices of the polytope $P_{G, 3}$ associated to a finite abelian group $G$ and a tripod are exactly the points $x(G_1, G_2,G_3)$ with $G_1+G_2+G_3=0$.

Let $L_{G, 3}$ be the lattice generated by vertices of $P_{G,3}$. Then
 $$L_{G,3}=\{x\in \mathbb Z^{3|G|}:\sum_{g,j}x_g^j\cdot g=0,
\forall \text{ } 1\le j,j' \le 3, \sum_g x_g^j=\sum_g x_g^{j'} \}.$$
where the first sum is taken in the group $G$.
\end{Theorem}

In addition, there is an implemented algorithm that can be used to obtain the vertices of $P_{G,3}$ for several abelian groups $G$. It can be found in \cite{marysia}.

\subsection{Reduction to simpler trees}

The toric fiber product of two homogeneous ideals belonging to two multigraded polynomial rings having the same multigrading is a construction due to Sullivant~\cite{seth}, which has interesting applications when working with polytopes $P_{G,3}$ (and, even more generally, $P_{G,n}$). More details may be consulted in \cite{seth, RM}.

We will use the following result due to Sullivant:

\begin{Theorem}(\cite[Theorem 3.10]{seth})\label{key}
The polytope associated to any trivalent tree $\mathcal{T}$, the polytope $P_{G, \mathcal{T}}$ can be expressed as the fiber product of polytopes associated to the tripod, $P_{G, 3}$.
More generally, the polytope associated to any tree $T$, the polytope $P_{G,T}$ can be expressed as the fiber product of polytopes associated to the $n$-claw tree, $P_{G,n}$.
\end{Theorem}


\section{Non-normal tripods}\label{non-normal}

To show that a polytope $P$ is not normal it is sufficient to provide an example of a lattice point $x$, such that $x\in mk(P\cap L)$ and $x\not\in k(P\cap L)$ for some integers $k,m$.

We will provide such an example for all abelian groups $G$ with the order being an odd number greater than 12. Moreover, we will always have $m=2$.




Let $\Gamma=(V,E)$ be a cubic graph whose edges are colored blue, yellow, and red, such that no two adjacent edges have the same color. For any vertex $v$ of $\Gamma$ let us denote by $e_B(v)$, ($e_Y(v)$, $e_R(v)$) the blue (yellow, red) edge adjacent to the vertex $v$.  

Let us call a function $f:E\rightarrow G$ \emph{good} if for any vertex $v$ of $\Gamma$ we have $$f(e_B(v))+f(e_Y(v))+f(e_R(v))=0.$$  

To any good function $f$ we can associate a point $x(f)\in |V|(P_{G,3})$ in the following way: First, consider any vertex $v$ of $\Gamma$. To this vertex we associate a point $x(v,f):=x(f(e_B(v),f(e_Y(v),f(e_R(v))$ which is a vertex of $P_{G,3}$ since $f$ is good. Then we simply define $$x(f):=\sum_{v\in V} x(v,f).$$

Note that the point $x(f)/2$ is still a lattice point. Indeed, all coordinates of $x(f)$ are even since every edge is counted twice - once for both of its endpoints. Then the conclusion follows from the following auxiliary and more general result:

\begin{Lemma} \label{even-lattice}
Let $x\in L_{G,n}$ be a point such that all coordinates of $x$ are even. Then $x/2 \in L_{G,n}$.    
\end{Lemma}
\begin{proof}
Since $x\in L_{G,n}$ we have $$0=\sum_{i=1}^n\sum_{g\in G\setminus \{0\}} x_g^i\cdot g=2\cdot \sum_{i=1}^n\sum_{g\in G\setminus \{0\}} (x_g^i/2)\cdot g.$$
Since $G$ has odd order, the only solution to the equation $g+g=0$ in $G$ is $g=0$. It follows that $$\sum_{i=1}^n\sum_{g\in G\setminus \{0\}} (x_g^i/2)\cdot g=0$$ 
and, hence, $x/2\in L_{G,n}$.
\end{proof}

If the point $x(f)$ can not be written as the sum of $|V|/2$ vertices of $P$, it is the point that proves the non-normality of $P_{G,3}$. Thus, it remains to show that there exists a graph $\Gamma$ and a good function $f:E\rightarrow G$ such that the point $x(f)$ has this property.

\begin{Lemma}\label{condition}
Let $\Gamma$ be a cubic 3-colorable graph which is not bipartite and $f:E\rightarrow G$ be a good function with the following property: if $f(e_1)+f(e_2)+f(e_3)=0$ for a blue edge $e_1$, yellow edge $e_2$ and red edge $e_3$, then $e_1=e_B(v),e_2=e_Y(v),e_3=e_R(v)$ for some vertex $v$ of $\Gamma$.
Then $x(f)/2$ can not be written as a sum of $|V|/2$ vertices of $P_{G,3}$.
\end{Lemma}
\begin{proof}
For the beginning, let $m=|V|/2$.

We claim that $f(e)\neq f(e')$ for any pair of edges of the same color. Assume for contradiction that $f(e)=f(e')$ for some pair of edges of the same color. Let $e_1,e_2$ be edges that share one vertex with edge $e$. Then $0=f(e)+f(e_1)+f(e_2)=f(e')+f(e_1)+f(e_2)$.
It follows that also all edges $e',e_1,e_2$ all share one vertex, and therefore $e=e'$.

This means that in the $G$-presentation of $x$ all multisets have $m$ different elements.

Assume that $x/2=p_1+p_2+\dots+p_m$, where $p_i$ are vertices of $P_{G,n}$. Then $p_i=x(f(e_1^i),f(e_2^i),f(e_3^i))$ for suitable edges $e_j^i$ of the graph $\Gamma$.
From the condition from the lemma statement it follows that the edges $e_1^i, e_2^i,e_3^i$ are adjacent to one vertex $v_i$ of $\Gamma$, i.e. $p_i=x(v_i,f)$. 

Since $x/2=p_1+p_2+\dots+p_m=x(v_1,f)+\dots+x(v_m,f)$, it must be true that every edge is adjacent to exactly one of the vertices $v_1,\dots,v_m$. However, this implies that the graph $\Gamma$ is bipartite, with one partition being $v_1,\dots,v_m$. Contradiction.
\end{proof}

Now it is clear how to find the good function $f$ for which $x(f)/2\notin |V|/2(P_{G,3}\cap L_{G,3})$. We just need to find a graph $\Gamma$ and a function $f$ which satisfy the property from Lemma~\ref{condition}. If $G$ is large enough, one can hope that there must be a suitable choice of a good function $f$. However, this would lead to a large bound on the order of $G$. Thus, we will try to weaken the condition from Lemma~\ref{condition}:

\begin{Lemma}\label{condition2}
Let $\Gamma$ be a cubic 3-colorable graph and $T=\{t_B,t_Y,t_R\}$ be blue, yellow, and red edge in $\Gamma$ that form a triangle. Let $f:E\rightarrow G$ be a good function with the following properties:
\begin{enumerate}[(i)]
\item $f(t_B)+f(t_Y)+f(t_R)\neq 0$,
\item $f(t)\neq f(e)$, for any different edges $t\in T, e\in E$ of the same color, 
\item $f(t)+f(e_1)+f(e_2)\neq 0$ for any edges $t\in T, e\in E\setminus T$ of different color.
\end{enumerate}
Then $x(f)/2$ can not be written as a sum of $|V|/2$ vertices of $P_{G,n}$.
\end{Lemma}
\begin{proof}
Analogously as in the proof of Lemma~\ref{condition}, we denote $m=|V|/2$ and we assume that $x(f)/2=p_1+\dots+p_m$, where $p_i=x(f(e_1^i),f(e_2^i),f(e_3^i))$ for suitable edges of the graph $\Gamma$. Note that the second condition implies $$x_f(t_B)^1+x_f(t_Y)^2+x_f(t_R)^3=2+2+2=6$$
which means that for the point $y:=x(f)/2$ we obtain that
$$y_f(t_B)^1+y_f(t_Y)^2+y_f(t_R)^3=1+1+1=3.$$ For any point $p=x(f(e_1),f(e_2),f(e_3))$, we have that $$p_f(t_B)^1+p_f(t_Y)^2+p_f(t_R)^3\not\in\{1,3\}.$$
This is a consequence of the first and the third condition. This implies that for all $p_i$, the sum $p_f(t_B)^1+p_f(t_Y)^2+p_f(t_R)^3$ is even. However, this is a contradiction with $p_1+\dots+p_m=y$, since the sum of even numbers can never be an odd number.
\end{proof}

Now we provide an example of graph $\Gamma$ and good functions $f$. Let $\Gamma$ be the following cubic graph on 12 vertices.
\begin{figure}[H]
  \includegraphics[width=\linewidth, scale=0.4, width=4.5cm]{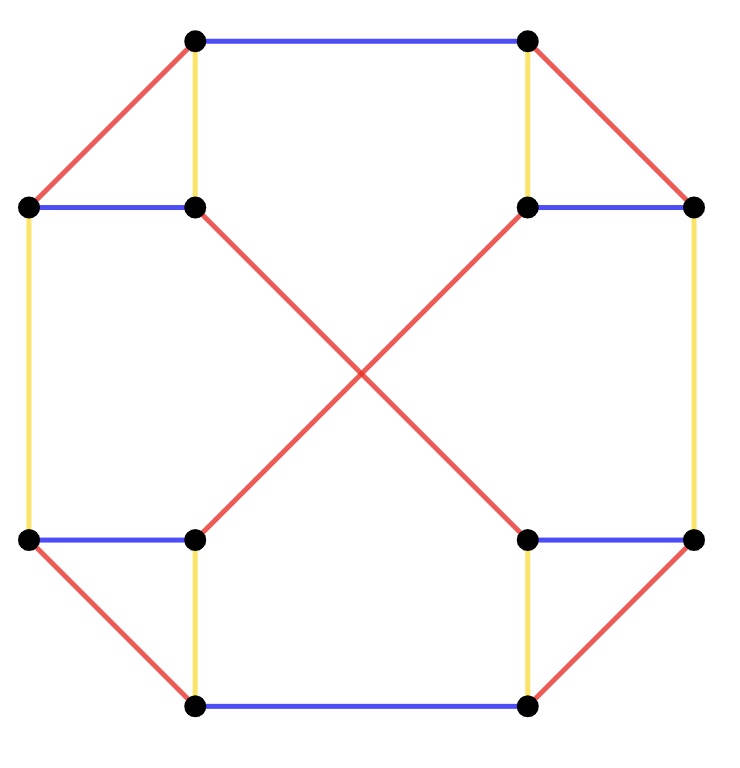}
  \caption{}
  \label{graf}
\end{figure}

First, we note that a good function is uniquely determined by the values of six edges that do not lie in any triangle. We denote the values of these edges by $h_1,\dots,h_6\in G$, as in Figure~\ref{graf2}. Then the value of the other edges is determined by expressions as in the figure. Note that, since $|G|$ is odd, the function $(1/2)\cdot g$ is well-defined for any element $g$.

To see, that the values of other edges are, in fact, determined by $h_1,\dots,h_6$, let us consider the edges $e_B,e_Y,e_R$ which form the left upper triangle. Since $f$ is good, we must have 
$$ f(e_B)+f(e_Y)+h_3=f(e_B)+h_2+f(e_R)=h_1+f(e_Y)+f(e_R)=0.$$

By summing up the first two equations and subtracting the last one, we obtain

$$2f(e_B)+h_3+h_2-h_1=0 \Leftrightarrow f(e_B)=\frac12 \cdot (-h_1+h_2+h_3))$$

Analogously, we can determine the values of other edges. It is easy to check, that for any choice of $h_1\dots,h_6$ this defines a good function $f$.

\begin{figure}[h!]
  \includegraphics[width=\linewidth, scale=0.5, width=8cm]{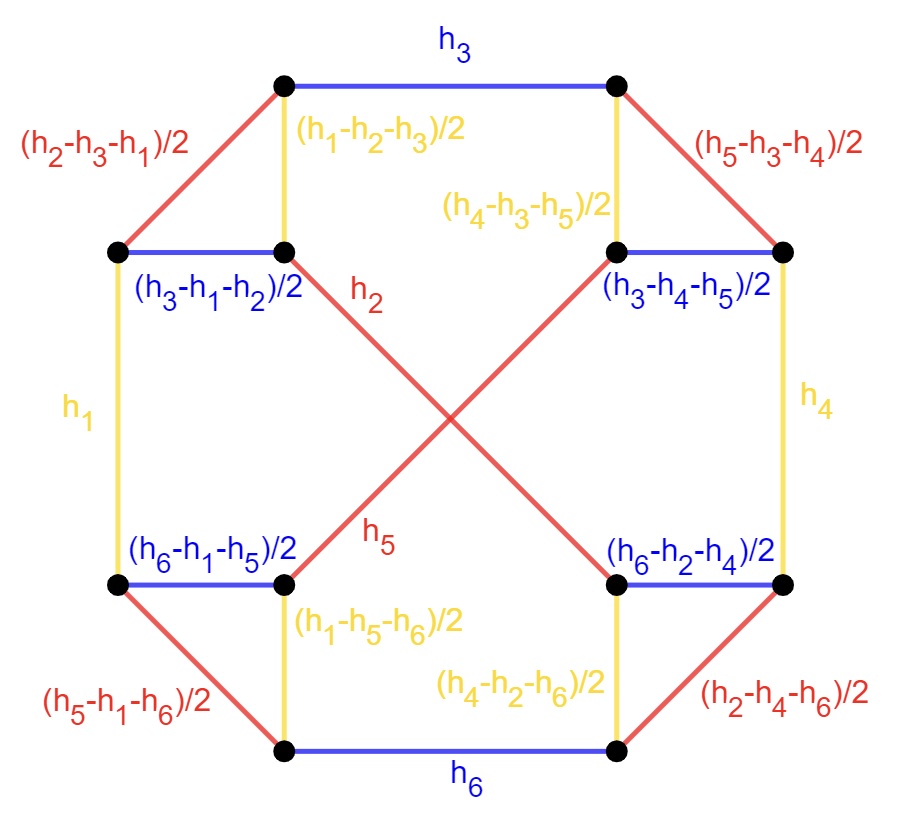}
  \caption{}
  \label{graf2}
\end{figure}

\begin{Theorem}\label{43}
Let $G$ be an abelian group of odd order which is greater than 43. Then $P_{G,3}$ is not normal.    
\end{Theorem}

\begin{proof}
Clearly, it is sufficient to provide an example of a good function $f$ that satisfies the conditions from Lemma~\ref{condition2}. For this, we need to provide the corresponding elements $h_1,\dots,h_6\in G$. Note that every condition from Lemma~\ref{condition2} requires that a certain linear form in $h_1,\dots,h_6$ must be different from 0. We want to show that the number of 6-tuples $(h_1,\dots,h_6)\in G^6$ that satisfies all conditions is positive.

Clearly, there are $|G|^6$ 6-tuples in $G^6$. Let $L(h_1,\dots,h_6)=l_1h_1+\dots+l_6h_6$ be a linear form such that $l_1,\dots,l_6\in\ZZ$ and $G$ is $l_i$-divisible for at least one index $i$. Then the number of 6-tuples for which $L(h_1,\dots,h_6)=0$ is $|G|^5$. To see this, simply pick a number $l_i$ such that $G$ is $l_i$-divisible. Then the element $h_i$ is uniquely determined by the choice of other group elements.

Now we simply compute the number of conditions that are in Lemma~\ref{condition2}. There is one linear form for $(i)$, there are $5\cdot 3=15$ conditions of type $(ii)$.

We note that some of the conditions of type $(iii)$ are redundant. Let $T$ be a triangle as in the statement of Lemma~\ref{condition2}. Consider the condition $f(t)+f(e_1)+f(e_2)\neq 0$, where the edges $e_1$ and $t$ are adjacent. Let $t'$ be the third edge adjacent to their common vertex. Clearly, $t'\in T$ and it is of the same color as edge $e_2$. Since $f(t)+f(e_1)+f(t')=0$, the condition $f(t)+f(e_1)+f(e_2)\neq 0$ is equivalent with $f(e_2)\neq f(t')$, which a condition from $(ii)$. Analogously, the condition $f(t)+f(e_1)+f(e_2)\neq 0$ is redundant also in the case where $e_2$ and $t$ are adjacent.

Similarly, consider the condition $f(t)+f(e_1)+f(e_2)\neq 0$ for adjacent edges $e_1$ and $e_2$. Let $e_3$ be the third edge adjacent to their common vertex. As, in the previous case, we have that $f(t)+f(e_1)+f(e_2)\neq 0$ is equivalent to $f(t)\neq f(e_3)$ which is, again, a condition of type $(ii)$.

Thus, it is sufficient to consider only such triples $t,e_1,e_2$ such that no two of them are adjacent. If the edge $t=t_B$, one can easily count that there are $9$ pairs of edges $e_1,e_2$ which are yellow and red, such that no two of the edges $t,e_1,e_2$ are adjacent. The same is true for $t=t_Y$ and $t=t_R$. Thus there are only $9+9+9=27$ (non-redundant) conditions of type $(iii)$.

One can easily check, that each of these conditions, indeed corresponds to a linear form $L(h_1,\dots,h_6)$ such that $G$ is $l_i$-divisible for at least one of its coefficients.

Therefore, there are at least $|G|^6-(1+15+27)|G|^5=|G|^5(|G|-43)>0$ 6-tuples of elements $(h_1,\dots,h_6)$ such that the corresponding good function $f$ satisfies all conditions from Lemma~\ref{condition2} which proves the desired result.

\end{proof}

In the previous result, we just show the existence of a good function $f$ without actually providing a concrete construction. An alternative approach is to find a good function $f$ simply by trying some 6-tuples $h_1,\dots,h_6$. This way, we are able to prove non-normality also for some smaller groups:

\begin{Theorem}\label{odd11}
Let $G$ be an abelian group of odd order which is greater than 11. Then $P_{G,3}$ is not normal.    
\end{Theorem}
\begin{proof}
If $|G|>43$, then $P_{G,3}$ is not normal, due to Proposition~\ref{43}. Here we provide an example of a good function $f$ for smaller groups $G$. Let $h_1=7$, $h_2=7$, $h_3=-4$, $h_4=-6$, $h_5=0$, $h_6=7$. This uniquely determines the good function $f:E\rightarrow \ZZ$ as in Figure~\ref{graf3}:

\begin{figure}[h!]
  \includegraphics[width=\linewidth, scale=0.5, width=6cm]{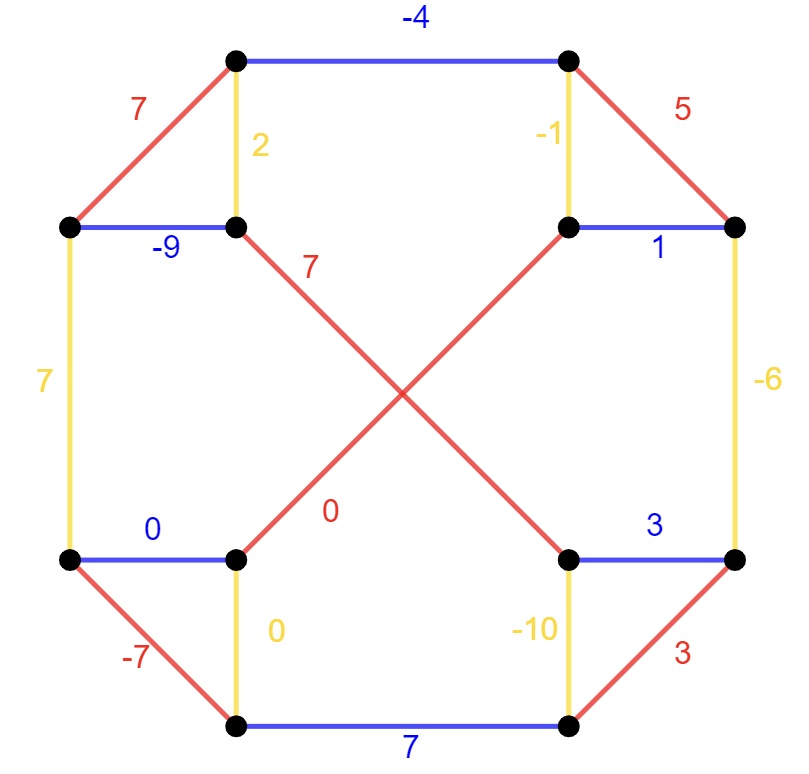}
  \caption{}
  \label{graf3}
\end{figure}

One can easily check that all sums $f(e_1)+f(e_2)+f(e_3)$ for edges of different colors are in the interval $[-26,21]$. Moreover, with the help of a computer, it is possible to check that none of these sums are equal to $23$ or $25$ and the only sums which are equal to $0$ are those for which $e_1,e_2,e_3$ have a common vertex.
This means that the function $f_n:E\rightarrow\ZZ_n$ which is a composition of $f$ with a natural map $i:\ZZ\rightarrow\ZZ_n$ satisfies the condition from Lemma \ref{condition} for all odd $n\ge 21$. By Lemma~\ref{condition}, the polytope $P_{G,3}$ is not normal when $G=\ZZ_n$ for any odd natural number $n>21$. 

We are left just with a few abelian groups, namely with  
$$G\in \{\ZZ_{13},\ZZ_{15},\ZZ_{17},\ZZ_{19},\ZZ_{21},\ZZ_5^2,\ZZ_3^3,\ZZ_9\times\ZZ_3 \}.$$

For each of them, we will provide a separate example of a good function $f$ that satisfies the condition from Lemma \ref{condition2}. The examples provided here by us are not unique, and one can check by a computer that they satisfy the required conditions. For each example we will provide just the corresponding elements $h_1,\dots,h_6$ which uniquely determined the good function $f$:

\begin{description}

\item[$G=\ZZ_{19}$] $h_1=1$, $h_2=1$, $h_3=7$, $h_4=15$, $h_5=0$, $h_6=1$.
\item[$G=\ZZ_{21}$] $h_1=3$, $h_2=3$, $h_3=1$, $h_4=9$, $h_5=0$, $h_6=3$.
\item[$G=\ZZ_{5}^2$] $h_1=(1,2)$, $h_2=(0,1)$, $h_3=(1,1)$, $h_4=(2,2)$, $h_5=(0,0)$, $h_6=(1,4)$.
\item[$G=\ZZ_{9}\times\ZZ_3$] $h_1=(5,2)$, $h_2=(2,2)$, $h_3=(3,1)$, $h_4=(1,0)$, $h_5=(0,0)$, $h_6=(5,2)$.
\end{description}

These four examples satisfy the (stronger) condition from Lemma~\ref{condition}.

\begin{description}
\item[$G=\ZZ_{3}^3$] $h_1=(1,1,0)$, $h_2=(0,0,1)$, $h_3=(0,2,2)$, $h_4=(2,2,0)$, $h_5=(0,0,0)$, $h_6=(1,1,0)$.

In this case, there is one additional triple of edges of different colors (except those that share a vertex) whose sum is equal to 0, namely, $h_4+h_5+h_6=0$. Still, this satisfies the conditions from Lemma~\ref{condition2} since none of these edges is contained in the upper left triangle $T$.
\item[$G=\ZZ_{17}$] $h_1=3$, $h_2=2$, $h_3=2$, $h_4=6$, $h_5=0$, $h_6=3$.

In this case, there is one additional triple of edges of different colors whose sum is equal to 0, namely, $(h_1-h_5-h_6)/2+h_2+(-h_4-h_5+h_3)/2=0$. Still, this satisfies the conditions from Lemma~\ref{condition2} since none of these edges is contained in the upper left triangle $T$.
\item[$G=\ZZ_{15}$] $h_1=3$, $h_2=10$, $h_3=1$, $h_4=1$, $h_5=0$, $h_6=8$.

In this case, there are two additional triples of edges of different colors whose sum is equal to 0. Namely, $(h_1-h_5-h_6)/2+h_2+(-h_4-h_5+h_3)/2=(h_4-h_2-h_6)/2+h_5+h_3=0$. Still, this satisfies the conditions from Lemma~\ref{condition2} since none of these edges is contained in the upper left triangle $T$.

\item[$G=\ZZ_{13}$] $h_1=1$, $h_2=1$, $h_3=1$, $h_4=9$, $h_5=0$, $h_6=1$.

In this case, $h_3=h_6$ creates 4 additional triples of edges of different colors whose sum is equal to 0. Still, each of them contains either 0 or two edges from the upper left triangle $T$, so this also satisfies conditions from Lemma~\ref{condition2}.

\end{description}
\end{proof}

We consider now the abelian group $G=\ZZ_{11}$. For this group, we have not found a good function that satisfies conditions of Lemma~\ref{condition2} on the graph $\Gamma$ which was used in the previous examples. We have not checked all good functions so it is possible that such a function exists, even though we strongly suspect that it does not. However, we manage to find a larger graph $\Gamma_2$ with a good function, that satisfies the properties from Lemma~\ref{condition2}.

\vspace{.1in}
\begin{Theorem}\label{11}
 The polytope $P_{\ZZ_{11},3}$ is not normal.   
\end{Theorem}
\begin{proof}
Consider the graph $\Gamma_2=(V_2,E_2)$ and the good function $f:E_2\rightarrow \ZZ_{11}$ displayed in Figure~\ref{graf11}:

\begin{figure}[H]
  \includegraphics[width=\linewidth, scale=0.5, width=8.5cm]{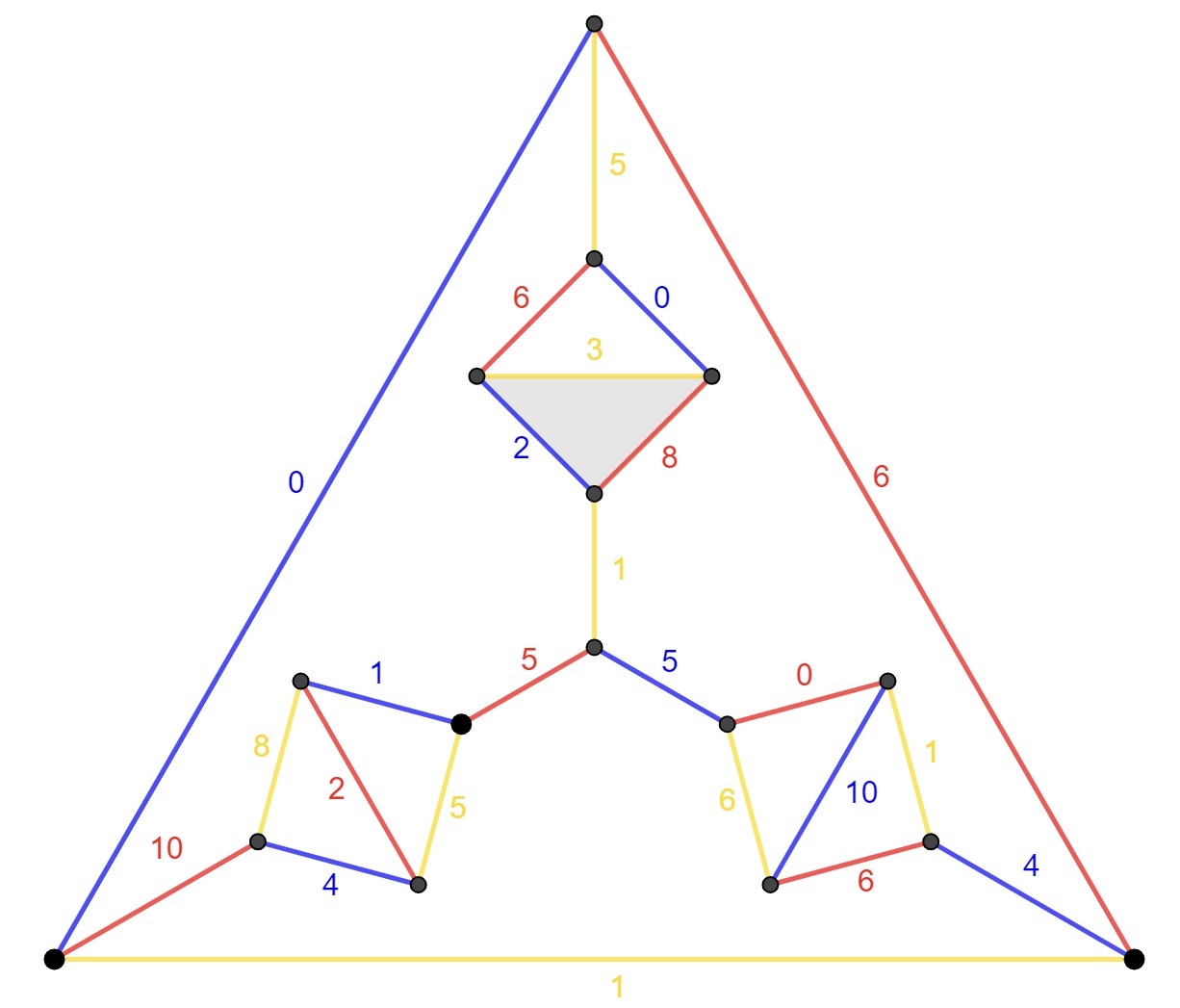}
  \caption{}
  \label{graf11}
\end{figure}

The triangle $T$ is the gray triangle from the picture, i.e. edges with values $2,3,8$.
Again, one can check by computer, or even by hand, that this satisfies the conditions of Lemma \ref{condition2} and, therefore, $P_{\ZZ_{11},3}$ is not normal.
\end{proof}

\begin{Remark}
The non-decomposable point from the last proof is the point $$x(\{0,0,1,2,4,4,5,10\},\{1,1,1,3,5,5,6,8\},\{0,2,5,6,6,6,8,10\}).$$ It is not difficult to find all triples of elements $(h_1,h_2,h_3)$, where $h_i$ is from the $i$-th multiset, that satisfy $$h_1+h_2+h_3=0 \wedge (h_1=2 \vee h_2=3 \vee h_3=8).$$
Indeed, the only such triples are $(2,3,6), (2,1,8), (0,3,8)$ which immediately shows that this point is non-decomposable. Despite the fact, we derive this example from the graph, this demonstrates what is happening just in terms of elements of $\ZZ_{11}$.
\end{Remark}

\section{Computational results}\label{comp}

In this section, we present a computational way to check if the polytopes associated to a tripod and the groups $\ZZ_5$ and $\ZZ_7$ are normal. In the next computational result, we give a positive answer to this question:

\begin{Computation}\label{5and7} The polytope $P_{G,3}$ associated to the tripod and any of the groups $G~\in~\{\ZZ_5, \ZZ_7\}$ is normal. Hence, the algebraic varieties representing these models are projectively normal.
\end{Computation}

Here we present the computational method we used. For obtaining the vertices of the polytopes $P_{G,3}$ where $G\in \{\ZZ_5, \ZZ_7\}$, we use the following code in Macaulay2~\cite{M2}, making use of the package ``Phylogenetic Trees". 

\begin{verbatim}
loadPackage "PhylogeneticTrees" 
n=7;
g=1_(ZZ/n);
G=for i from 0 to n-1 list i*g;
B=for i from 0 to n-1 list {G#i};
M=model(G,B,{});
T=leafTree(3,{});
A=submatrix'(phyloToricAMatrix(T,M),{0,n,2*n},);
I=id_(QQ^(n-1));
PP=for i from 1 to n-1 list n-i;
TM=inverse((I|(-1)*I|0*I)||(I|0*I|I)||((transpose(matrix{PP})+
  submatrix(I,,{0}))|submatrix'(I,,{0})|0*I|0*I));
LL=for i from 0 to n^2-1 list 1;
AA=(transpose(matrix{LL}))|transpose(A)*TM;
entries(AA)
\end{verbatim}

Note that we want to check the normality of the polytope in the lattice spanned by its vertices and not in the lattice $\ZZ^{3(|G|-1)}$. Thus, we write the coordinates of vertices of $P_{G,3}$ in the basis of $L_{G,3}$,
hence this requires the last part of the presented code. More precisely, we use the following basis for the lattice $L_{\ZZ_n,3}$:
\begin{itemize}
\item $x(\{g\},\emptyset,\emptyset)-x(\emptyset,\{g\},\emptyset)$ for all $g\in \ZZ_n$
\item $x(\{g\},\emptyset,\emptyset)-x(\emptyset,\emptyset,\{g\})$ for all $g\in \ZZ_n$
\item $n\cdot x(\{1\},\emptyset,\emptyset)$
\item $x(\{\underbrace{1,1,\dots,1}_{(n-i)\text{-times}},i\},\emptyset,\emptyset)$ for all $2\le i\le n-1$.
\end{itemize}

After obtaining the vertices of the polytope, we use Polymake~\cite{polymake} in order to check its normality.

We also use a similar computational approach to check the normality for groups $\ZZ_9$ and $\ZZ_3^2$. However, this time the computation resulted in a negative answer:

\begin{Computation}\label{9and3x3} The polytope $P_{G,3}$ associated to the tripod and any of the groups $G~\in~\{\ZZ_9, \ZZ_3^2\}$ is not normal. Hence, the algebraic varieties representing these models are not projectively normal.
\end{Computation}

For the group $\ZZ_3\times \ZZ_3$ we use the following code in Macaulay2 for obtaining the vertices, and, again, for checking normality we use Polymake.

\begin{verbatim}
loadPackage "PhylogeneticTrees" 
g1={1_(ZZ/3),0_(ZZ/3)};
g2={0_(ZZ/3),1_(ZZ/3)};
G={0*g1,g1,2*g1,g2,g1+g2,2*g1+g2,2*g2,g1+2*g2,2*g1+2*g2};
B=for i from 0 to 8 list {G#i};
M=model(G,B,{});
T=leafTree(3,{});
A=phyloToricAMatrix(T,M);
AAA=submatrix'(A,{0,8,16},);
I=id_(QQ^8);
PP=matrix{{3,0,0,0,0,0,0,0},{0,0,3,0,0,0,0,0},{1,1,0,0,0,0,0,0},
{1,0,1,-1,0,0,0,0},{2,0,1,0,-1,0,0,0},{0,0,1,0,0,1,0,0}, 
{1,0,2,0,0,0,-1,0}, {2,0,2,0,0,0,0,-1}};
TM=inverse((I|(-1)*I|0*I)||(I|0*I|I)||(PP|0*I|0*I));
ATM=transpose(AAA)*TM;
LL=for i from 0 to 80 list 1;
AA=(transpose(matrix{LL}))|ATM;
L=entries(AA)
\end{verbatim}

For the group $\ZZ_9$ one can use the program \cite{marysia} for getting the vertices of the polytope $P_{\ZZ_9, 3}$, and then we proceed as for the previous groups.

\section{Classification}\label{class}

In this section, we present the main result. Namely, we provide a complete classification of normal group-based phylogenetic varieties for tripods, and more generally, for trivalent trees.

\begin{Theorem}\label{main}
Let $G$ be an abelian group. Then the polytope $P_{G,3}$ associated to a tripod and the group $G$ is normal if and only if $G\in\{\ZZ_2, \ZZ_3, \ZZ_2\times \ZZ_2, \ZZ_4, \ZZ_5, \ZZ_7\}$.
\end{Theorem}

\begin{proof}
By Theorem~\ref{43}, $P_{G,3}$ is not normal, for any abelian group $G$ which has odd cardinality greater than 43. The same is true, by Theorem~\ref{odd11}, the same is true, for any abelian group $G$ which has odd cardinality between 12 and 43, and, by Theorem~\ref{11}, if $|G|=11$. Now, by \cite[Proposition 2.1]{RM}, $P_{G,3}$ is not normal, for any abelian group of even cardinality greater or equal than 6. 
The polytopes $P_{G,3}$ are normal when $G\in \{\ZZ_5, \ZZ_7\}$, by Computation~\ref{5and7}, and non-normal when $G\in \{\ZZ_9, \ZZ_3\times \ZZ_3\}$, by Computation~\ref{9and3x3}. If $G=\ZZ_4$, the polytope $P_{G,3}$ is normal by computations shown in \cite{mateusz}. When $G=\ZZ_2\times \ZZ_2$, the polytope corresponding to the 3-Kimura model is normal by \cite{martinko}, even for any tree.  If $G=\ZZ_3$, the corresponding polytope is normal by \cite[Theorem 2.3]{RM}, for any tree. When $G=\ZZ_2$, the polytope corresponding to the Cavender-Farris-Neyman is normal, by \cite{bw} and \cite{martinko}, for any tree.
    
\end{proof}

As a consequence of Theorem~\ref{main}, and Sullivant's result Theorem~\ref{key}, we obtain that:

\begin{Corollary}
Let $G$ be an abelian group. Then the polytope $P_{G,\mathcal{T}}$ associated to any trivalent tree and the group $G$ is normal if and only if $G\in\{\ZZ_2, \ZZ_3, \ZZ_2\times \ZZ_2, \ZZ_4, \ZZ_5, \ZZ_7\}$.
\end{Corollary}

Therefore, in terms of the associated toric varieties, we get:

\begin{Corollary}
Let $G$ be an abelian group. Then the group-based phylogenetic variety $X_{G,3}$ associated to a tripod and the group $G$ is projectively normal if and only if $G\in\{\ZZ_2, \ZZ_3, \ZZ_2\times \ZZ_2, \ZZ_4, \ZZ_5, \ZZ_7\}$.
Moreover, this result holds for the group-based phylogenetic variety $X_{G,\mathcal{T}}$ associated to any trivalent tree.
\end{Corollary}

\begin{Remark} Let $G$ be an abelian group.
As the non-normality of polytopes associated to tripods implies the non-normality of polytopes associated to any tree, the groups from Theorem~\ref{main} are the only candidates to give rise to projectively normal toric varieties associated to an arbitrary tree. In addition, for $\ZZ_2$, $\ZZ_2\times \ZZ_2$ and $\ZZ_3$, the corresponding phylogenetic are known to be normal for any tree, hence, it remains to understand only the normality when the groups are $\ZZ_4, \ZZ_5$ and $\ZZ_7$ for any tree. 
\end{Remark}

We also used the computer to check the normality for $4$-claw tree and the above groups, and it turns out that the polytopes are normal. We suspect that these groups give rise to projectively normal phylogenetic varieties for any tree and we propose the following:

\begin{Conjecture}
Let $G$ be an abelian group and $T$ a tree. Then the group-based phylogenetic variety $X_{G,T}$ is projectively normal if and only if $G\in\{\ZZ_2, \ZZ_3, \ZZ_2\times \ZZ_2, \ZZ_4, \ZZ_5, \ZZ_7\}$.
\end{Conjecture}

\section*{Acknowledgement} RD was supported by the Alexander von Humboldt Foundation and by a grant of the Ministry of Research, Innovation and Digitization, CNCS - UEFISCDI, project number PN-III-P1-1.1-TE-2021-1633, within PNCDI III.

MV was supported by Slovak VEGA grant 1/0152/22.


\vspace{.3in}
\begin{thebibliography}{}
\vspace{.1in}

 \bibitem{brunsg} W.~Bruns, J.~Gubeladze, \textit{Polytopes, Rings and K-theory}, Springer Monogr. Math., 2009.\\

 \bibitem{bw} Weronika~Buczy\'nska, Jaros{\l}aw~Wi\'sniewski, \textit{On geometry of binary symmetric models of phylogenetic trees}, J. Eur. Math. Soc., {\bf 9} (3) (2007), 609--635.\\

 \bibitem{cox} David~Cox, John~Little, Hal~Schenck, \textit{Toric varieties}, Amer. Math. Soc., 2011.\\
 
\bibitem{RM} Rodica ~Dinu, Martin~Vodička, \textit{Gorenstein property for phylogenetic trivalent trees}, Journal of Algebra, {\bf 575} (2021), 233--255.\\

\bibitem{RM2} Rodica ~Dinu, Martin~Vodička, \textit{Phylogenetic degrees for claw trees}, arXiv:2303.16812, preprint, 2023.\\

  \bibitem{marysia} Maria ~Donten-Bury's WebPage, \textit{Programs C++ for analyzing polytopes associated with phylogenetic models}, 
  \newblock Available at \url{https://www.mimuw.edu.pl/~marysia/polytopes/}.\\



\bibitem{sz} P.L.~Erd\"os, M.A.~Steel, L.A.~Sz\'{e}kely, \textit{Fourier calculus on evolutionary trees}, Adv. Appl. Math.,  {\bf 14} (2) (1993), 200--210.\\
 
 \bibitem{erss} Nicholas~Eriksson, Kristian~Ranestad, Bernd~Sturmfels, Seth~Sullivant, \textit{Phylogenetic algebraic geometry}, Projective varieties with Unexpected Properties, Siena, Italy (2004), 237--256.\\

  \bibitem{evans} Steven N.~Evans, Terrence ~P.~Speed, \textit{Invariants of some probability models used in phylogenetic inference}, Ann. Stat., {\bf 21} (1) (1993), 355--377.\\
 
 \bibitem{fels} Joseph~Felsenstein, \textit{Inferring phylogenies}, Sinauer Associates, Inc., Sunderland, 2003.\\

 \bibitem{fulton} William Fulton, \textit{Introduction to toric varieties}, {\bf 131}, Princeton university press, 1993.\\

 \bibitem{polymake} Ewgenij~Gawrilow, Michael~Joswig, \textit{Polymake: a framework for analyzing convex polytopes. Polytopes-combinatorics and computation}, (Oberwolfach, 1997), {\bf 43-73}, DMV Sem., 29, Birkh\" auser, Basel, 2000.\\

 \bibitem{M2} Daniel~R. Grayson and Michael~E. Stillman, \textit{Macaulay2}, a software system for research in algebraic geometry.
\newblock Available at \url{http://www.math.uiuc.edu/Macaulay2/}.\\

\bibitem{HP} Michael D.~Hendy, David~Penny, \textit{A framework for the quantitative study of evolutionary trees}, Systematic zoology, {\bf 38}(4)(1989), 297--309.\\

\bibitem{binomialideals} J\"urgen~Herzog, Takayuki~Hibi, Hidefumi~Ohsugi, \textit{Binomial ideals}, Grad. Texts in Math., {\bf 279}, Springer, 2018.\\

\bibitem{kimura} Motoo~Kimura, \textit{Estimation of evolutionary distances between homologous nucleotide sequences}, Proceedings of the National Academy of Sciences, {\bf 78} (1) (1981), 454--458.\\

\bibitem{manon} Christopher~Manon, \textit{Coordinate rings for the moduli stack of $SL_2(C)$ quasi-parabolic principal bundles on a curve and toric fiber products}, J. Algebra, {\bf 365}(2012), 163--183.\\

 \bibitem{h-rep} Marie~Mauhar, Joseph~Rusinko, Zoe~Vernon, \textit{H-representation of the Kimura-3 polytope for the m-claw tree}, SIAM J. Discrete Math., {\bf 31} (2) (2017), 783--795.\\

\bibitem{matJCTA} Mateusz~Micha{\l}ek, \textit{Constructive degree bounds for group-based models}, Journal of Combinatorial Theory, Series A, {\bf 120}(7)(2013), 1672--1694.\\

\bibitem{mateusz} Mateusz~Micha{\l}ek, \textit{Geometry of phylogenetic group-based models}, J. Algebra, (2011), 339--356.\\


\bibitem{mateuszKimura} Mateusz~Micha{\l}ek, Emanuele~Ventura, \textit{Phylogenetic complexity of the Kimura 3-parameter model}, Advances in Mathematics, {\bf 343}(2019), 640--680.\\


\bibitem{sturmfels} B.~Sturmfels, \textit{Gr\"obner bases and convex polytopes}, Amer. Math. Soc., Providence, RI, 1995.\\

\bibitem{ss} Bernd~Sturmfels, Seth~Sullivant, \textit{Toric ideals of phylogenetic invariants}, J. Computat. Biol., {\bf 12} (2005), 204--228.\\

\bibitem{seth} Seth~Sullivant, \textit{Toric fiber product}, J. Algebra,  {\bf 316} (2) (2007), 560--577.\\

\bibitem{martinko} Martin~Vodi\v{c}ka, \textit{Normality of the Kimura 3-parameter model}, SIAM Journal on Discrete Mathematics, {\bf 35}(3) (2021), 1536--1547.


\end{thebibliography}
\end{document}